\documentclass[12pt]{article}

\usepackage{amssymb}
\usepackage{amsthm}
\usepackage{amsmath}
\usepackage{graphicx}
\usepackage{fullpage}
\usepackage{color}
 \numberwithin{equation}{section}

\theoremstyle{plain}
\newtheorem{thm}{Theorem}[section]

\newtheorem{cor}[thm]{Corollary}

\newtheorem{lem}[thm]{Lemma}
\newtheorem{prop}[thm]{Proposition}

\newtheorem{conj}[thm]{Conjecture}

\theoremstyle{definition}
\newtheorem{defn}[thm]{Definition}
\newtheorem{ex}[thm]{Example}

\theoremstyle{remark}

\newtheorem{rem}[thm]{Remark}

\newcommand{\N}{\mathbb{N}}
\newcommand{\R}{\mathbb{R}}

\newcommand{\A}{\mathcal{A}}

\newcommand{\I}{\infty}
\newcommand{\bp}{\begin{proof}[\ensuremath{\mathbf{Proof}}]}
\newcommand{\ep}{\end{proof}}
\newcommand{\id}{\text{id}_{\R^d}}

\newcommand{\cU}{{\cal U}}
\newcommand{\cV}{{\cal V}}
\newcommand{\cMp}{{\cal M}_p}
\newcommand{\cPp}{{\cal P}_p(\R^d)}

\begin{document}


\title{Infinite horizon value functions in the Wasserstein spaces}

\author{Ryan Hynd\footnote{Partially supported by NSF grants DMS-1004733 and DMS-1301628.}\;
and Hwa Kil Kim\footnote{The work of H. K. Kim is supported by the
BK21 Plus SNU Mathematical Sciences Division. }}

\maketitle

\begin{abstract}
We perform a systematic study of optimization problems in the Wasserstein spaces that are analogs of
infinite horizon, deterministic control problems.  We derive necessary conditions on action minimizing paths and present a sufficient condition for their existence.
We also verify that the corresponding generalized value functions are a type of viscosity solution of a time independent, Hamilton-Jacobi equation in the space of probability measures.  Finally, we prove a special case of a  conjecture involving the subdifferential of generalized value functions and their relation to action minimizing paths.
\end{abstract}




\section{Introduction}
In this paper, we consider analogs of the classical value functions
\begin{equation}\label{classVal}
u(x)=\inf\left\{\int^\I_0e^{-\delta t}\left(\frac{1}{p}|\dot\gamma(t)|^p -V(\gamma(t)) \right)dt: \gamma\in AC_{p,\delta}(\R^d),\;\gamma(0)=x\right\}
\end{equation}
in the space of probability measures. In the above formula, $x\in \R^d$, $p\in (1,\infty)$, $\delta\in (0,\infty)$, $V\in C^1(\R^d)$, and $AC_{p,\delta}(\R^d)$ is the collection
of locally absolutely continuous paths $\gamma: [0,\infty)\rightarrow \R^d$ satisfying $\int^\I_0e^{-\delta t}|\dot\gamma(t)|^pdt<\infty$.  Under appropriate growth assumptions on $V$,
minimizers of $u(x)$ exist and satisfy the Euler-Lagrange ODE
\begin{equation}\label{EulLagDel}
\frac{d}{dt}\left(|\dot\gamma(t)|^{p-2}\dot\gamma(t)\right) = - \nabla V(\gamma(t))+\delta |\dot\gamma(t)|^{p-2}\dot\gamma(t), \quad t>0.
\end{equation}
\par  Another important fact is that $u$ in \eqref{classVal} can be characterized as a viscosity solution of the Hamilton-Jacobi equation (HJE)
\begin{equation}\label{ClassicalHJE}
\delta u+\frac{1}{q}|\nabla u|^q + V(x)=0, \quad x\in \R^d.
\end{equation}
Here and throughout, $1/q+1/p=1$. Moreover, it can be shown that $u$ is differentiable along minimizing paths and that a necessary and sufficient condition for a path to be minimizing  is
\begin{equation}\label{ClassicalGradFlow}
|\dot\gamma(t)|^{p-2}\dot\gamma(t)=-\nabla u(\gamma(t)), \quad t>0.
\end{equation}
The goal of this work is to extend some of these ideas to the space of probability measures.

\par In our prior study \cite{HyndKim}, we considered value functions on a finite horizon in the Wasserstein spaces. This current work addresses the analogous class of infinite horizon, {\it generalized value functions}
\begin{equation}\label{GenValueFun}
\cU(\mu)=\inf\left\{\int^\I_0e^{-\delta t}\left(\frac{1}{p}||\dot\sigma(t)||^p -\cV(\sigma(t)) \right)dt: \sigma\in AC_{p,\delta}(\cMp),\;\sigma(0)=\mu\right\}.
\end{equation}
In equation \eqref{GenValueFun}, the $p$th Wasserstein space $\cMp$ is
$$
\cPp:=\left\{\text{Borel probability measures $\mu$ on $\R^d$}: \int_{\R^d}|x|^pd\mu(x)<\infty\right\}
$$
equipped with the metric
\begin{equation}\label{WpDef}
W_p(\mu,\nu):=\inf\left\{ \left(\iint_{\R^d\times\R^d}|x-y|^pd\pi(x,y)\right)^{1/p}: \pi\in \Gamma(\mu,\nu)\right\}.
\end{equation}
The infimum in \eqref{WpDef} is taken over $\Gamma(\mu,\nu)$, the subcollection of ${\cal P}_p(\R^d\times\R^d)$ having first marginal $\mu$ and second marginal $\nu$. A wealth of information about $\cMp$ can be found in the references \cite{AGS,V}.

\par In \eqref{GenValueFun}, $t\mapsto ||\dot\sigma(t)||$ is the usual metric derivative of a locally $p$-absolutely
continuous path $\sigma: [0,\infty)\rightarrow \mathcal{M}_p$ (Definition 1.1.1 in \cite{AGS}). By Theorem 8.3.1 \cite{AGS},  there is a Borel vector field $v:\R^d\times[0,\infty)\rightarrow \R^d$ such that the continuity equation holds in the sense of distributions
$$
\partial_t\sigma + \nabla\cdot(\sigma v)=0, \quad \R^d\times[0,\infty).
$$
Moreover,
$$
||v(t)||_{L^p(\sigma(t))}=||\dot\sigma(t)|| \quad \text{a.e. $t>0$}.
$$
The mapping $v$ is known as the {\it minimal velocity} of $\sigma$ as $t\mapsto v(\cdot,t)$ is essentially uniquely determined and satisfies $||v(t)||_{L^p(\sigma(t))}\le ||w(t)||_{L^p(\sigma(t))}$ (for a.e. $t\ge 0$) for any other Borel field $w$ satisfying the continuity equation with $\sigma$.  For this work, an important subset of the space of locally $p$-absolutely continuous paths is $AC_{p,\delta}(\cMp)$ which additionally requires
\begin{equation}\label{vdeltapbound}
\int^\infty_0\int_{\R^d}e^{-\delta t}|v(x,t)|^pd\sigma_t(x)dt<\infty.
\end{equation}

\par We show that minimizing paths for $\cU$ and their corresponding minimal velocities satisfy the {\it Euler-Poisson system}
\begin{equation}\label{StaticEP}
\begin{cases}
\partial_t(\sigma |v|^{p-2}v) + \nabla\cdot (\sigma|v|^{p-2} v\otimes v) = -\sigma \nabla\cV(\sigma)  + \sigma\delta |v|^{p-2}v \\
\hspace{1.28in}\partial_t\sigma +\nabla\cdot(\sigma v)=0
\end{cases}
\end{equation}
in the sense of distributions on $\R^d\times(0,\infty)$.
We also present a result on the existence of minimizing paths provided the potential $\cV$ satisfies an appropriate growth condition and
is continuous with respect to the narrow topology.  To this end, we make specific use of the fact that the narrow topology on $\cPp$ can metrized by the L\'{e}vy-Prokhorov
metric $\Lambda$ (defined in \eqref{LPmetric}).

\begin{thm}\label{WeakExistenceThm}
Assume that $\cV\in C((\cPp,\Lambda))$ and
\begin{equation}\label{weakVBound}
|\cV(\mu)|\le \alpha \Lambda(\mu,\rho)^{2r}+\beta, \quad \mu\in \cMp
\end{equation}
for some $\alpha, \beta\in \R$, $1\le r<p$ and $\rho\in \cMp$. Then $\cU(\mu)$ has a minimizing path for each $\mu\in \cMp$.
\end{thm}

\par As in our previous paper \cite{HyndKim}, the main result of this work is that each generalized value functions is a type of viscosity solution of an appropriate Hamilton-Jacobi equation.
\begin{thm}\label{mainThm}
Assume
\begin{equation}\label{Vgrowth}
\cV(\mu)\le \alpha W_p(\mu,\varrho)^p+\beta,\quad \mu\in \cMp
\end{equation}
for some measure $\varrho\in \cMp$, $\alpha, \beta\in \R$.
 Then there is $\delta_0=\delta_0(p,\alpha)>0$ such that functional $\cU$ \eqref{GenValueFun} is a viscosity solution (Definition \ref{ViscDef}) of the HJE
\begin{equation}\label{NewHJE}
\delta \cU+\frac{1}{q}||\nabla\cU||^q_{L^q(\mu)} + \cV(\mu)=0, \quad \mu\in \cMp
\end{equation}
provided $\delta>\delta_0$.
\end{thm}

\par An important example occurs when $\cV$ is a {\it simple potential}:
\begin{equation}\label{SimpleV}
\cV(\mu)=\int_{\R^d}V(x)d\mu(x)
\end{equation}
for some $V\in C^1(\R^d)$.  We note that if $V$ is not uniformly bounded, $\cV$ will not in general be narrowly continuous (Remark 7.1.11 in \cite{AGS}). So it is not immediate that the
corresponding generalized value function $\cU$ will have minimizing paths. However, we show that this is generally the case and provide a formula for the generalized value function \eqref{GenValueFun} in terms of the
classical one \eqref{classVal}.

\begin{prop}\label{SpecForm}
Assume $\cV$ is a simple potential and that $V(x)=O(|x|^r)$ as $|x|\rightarrow \infty$ for some $1\le r<p$. Then
\begin{equation}\label{SpecValueFun}
\cU(\mu)=\int_{\R^d}u(x)d\mu(x), \quad \mu \in \cMp.
\end{equation}
Moreover, there is also a Borel map $\Psi: \R^d\times [0,\infty)\rightarrow \R^d$ such that for each
$x\in \R^d$, $t\mapsto \Psi(x,t)$ is a minimizer for $u(x)$, and for each $\mu\in\cMp$,
$$
\sigma(t):= \Psi(t)_\#\mu, \quad t\in [0,\infty)
$$
is a minimizing path for $\cU(\mu)$.
\end{prop}
A corollary of the above proposition provides a connection between the classical HJE \eqref{ClassicalHJE} and the Euler-Poisson system \eqref{StaticEP}. We show
that if $\sigma$ is any minimizing trajectory for $\cU(\mu)$ with minimal velocity $v$, then
$$
|v(x,t)|^{p-2}v(x,t)=-\nabla u(x), \quad \sigma(t)\;\;\text{a.e.} \; x\in \R^d
$$
for Lebesgue almost every $t>0$. This identity can be seen at a heuristic level by differentiating the classical HJE \eqref{ClassicalHJE} and comparing the result to the Euler-Possion system \eqref{StaticEP}, when
$\cV$ satisfies \eqref{SimpleV}.

\par There is a growing literature on value functions in the space of probability measures \cite{Feng, G, G2, GS, HyndKim} and more generally in metric spaces \cite{AF, Brasco, Crandall, Goz, Ober}. However, the bulk of these efforts have been restricted to optimization problems on a {\it finite time horizon}; as mentioned, we consider model infinite horizon problems in this paper.  Regarding the Wasserstein space, the main technical differences with our work on finite horizon required a nonstandard, weighted Poincar\'{e} inequality (Lemma \ref{1DPoincare}) and a new probabilistic representation of $AC_{p,\delta}(\cMp)$ paths (Theorem \ref{ProbRep}).

\par One aspiration we have is to further develop the theory of action minimizing paths. We conjecture that a necessary and sufficient condition for minimizing trajectories is that they satisfy a gradient flow condition analogous to \eqref{ClassicalGradFlow}; see Conjecture \ref{GradConj}. Unfortunately, we are only able to verify a particular case in \eqref{WeakConj}.  Another open problem is to verify whether  generalized value functions are unique as viscosity solutions of the HJE \eqref{NewHJE}. These are both areas of ongoing research.


\par This paper is organized as follows.  In section \ref{PropSec}, we deduce several important properties of generalized
value functions including various continuity assertions. In section \ref{MinSec}, we study minimizers of $\cU$, derive the Euler-Poisson system and verify Theorem \ref{WeakExistenceThm}.  We study simple potentials and prove Proposition \ref{SpecForm} in section \ref{SimpleSec}; and in section \ref{HJESec}, we verify Theorem \ref{mainThm} which asserts the viscosity solution property of $\cU$.  We thank the Universitat Polit\`{e}cnica de Catalunya, the Math Science Research Institute and Nathan and Angela George for their hospitality during the writing of this paper.



\section{Various properties}\label{PropSec}
In this section, we will deduce various characteristics of generalized value functions $\cU$ \eqref{GenValueFun}. In particular, we will show that under the appropriate hypotheses that $\cU$ obeys a dynamic programming principle and is continuous. First, we will start by showing that $\cU$ is well defined and satisfies simple pointwise bounds provided $\delta$ is large enough. To this end, we will need a type of weighted Poincar\'{e} inequality.
\begin{lem}
Assume $p\in [1,\infty)$. Then
\begin{equation}\label{1DPoincare}
\left(\int^\I_0e^{-\delta t}|u(t)-u(0)|^pdt\right)^{1/p}\le \frac{p}{\delta}\left(\int^\I_0e^{-\delta t}|\dot u(t)|^pdt\right)^{1/p}
\end{equation}
for all $\delta>0$ and $u\in AC_{\text{loc}}([0,\infty); \R)$.
\end{lem}
\begin{proof}  Assume $u(0)=0$, $p\in (1,\infty)$ and fix $T>0$. Integrating by parts and using H\"{o}lder's inequality gives
\begin{align*}
\int^T_0e^{-\delta t}|u(t)|^pdt & =  \left.-\frac{1}{\delta}|u(t)|^pe^{-\delta t}\right|^T_0 + \frac{p}{\delta}\int^T_0 |u(t)|^{p-2}u(t)\dot u(t) e^{-\delta t}dt\\
&= -\frac{1}{\delta}|u(T)|^pe^{-\delta T} + \frac{p}{\delta}\int^T_0 |u(t)|^{p-2}u(t)e^{-\frac{\delta}{q} t}\;\dot u(t) e^{-\frac{\delta}{p} t}dt \\
&\le  \frac{p}{\delta} \left(\int^T_0e^{-\delta t}|u(t)|^pdt\right)^{1-1/p} \left(\int^T_0e^{-\delta t}|\dot u(t)|^pdt\right)^{1/p}.
\end{align*}
Thus, \eqref{1DPoincare} holds by sending  $T\rightarrow \infty$. Likewise, for $p=1$
\begin{align*}
\int^T_0e^{-\delta t}|u(t)-u(0)|dt &\le \int^T_0e^{-\delta t}\left(\int^t_0|\dot u(s)|ds\right)dt \\
&= \left.-\frac{1}{\delta}\left(\int^t_0|\dot u(s)|ds\right)e^{-\delta t}\right|^T_0 + \frac{1}{\delta}\int^T_0 e^{-\delta t}|\dot u(t)|dt\\
&\le \frac{1}{\delta}\int^T_0 e^{-\delta t}|\dot u(t)|dt.
\end{align*}
Again we conclude by sending $T\rightarrow \infty$.
\end{proof}
An immediate corollary of Lemma \ref{1DPoincare} is as follows.

\begin{cor}
Assume $p\in [1,\infty)$. Then
\begin{equation}\label{PoincareNewWeight}
\left(\int^\I_0e^{-\delta t}W_p(\sigma(t),\sigma(0))^pdt\right)^{1/p}\le \frac{p}{\delta}\left(\int^\I_0e^{-\delta t}||\dot\sigma(t)||^pdt\right)^{1/p}
\end{equation}
for each locally p-absolutely continuous path $\sigma:[0,\infty)\rightarrow \cMp$.
\end{cor}
\begin{proof}
Define the function $u(t):=W_p(\sigma(t),u(0))$, for $t\in [0,\infty)$. The triangle inequality implies $|u(t_1)-u(t_2)|\le W_p(\sigma(t_1), \sigma(t_2))$; in
particular,  $u\in AC_{\text{loc}}([0,\infty); \R)$, and $|\dot u(t)|\le ||\dot\sigma(t)||$.  Repeating the proof of \eqref{1DPoincare}, we easily conclude \eqref{PoincareNewWeight}.
\end{proof}
\begin{lem}\label{BoundsLem}
Assume the growth condition \eqref{Vgrowth} and
\begin{equation}\label{cond18}
p(2p/\delta)^p\alpha <1.
\end{equation}
Then
\begin{equation}\label{BoundsU}
-\frac{1}{\delta}\left(\beta + 2^p\alpha W_p(\mu,\varrho)^p\right)\le \cU(\mu)\le -\frac{1}{\delta}\cV(\mu).
\end{equation}
In particular, $\cU(\mu)>-\infty$ for all large values of $\delta$.
\end{lem}
\begin{proof}
Let $\mu\in \cMp$ and $\sigma$ be admissible for $\cU(\mu)$. Using the growth assumption on $\cV$ and the weighted Poincare inequality
\begin{align*}
\int^\infty_0e^{-\delta t}\left(\frac{1}{p}||\dot\sigma(t)||^p -\cV(\sigma(t)) \right)dt & \ge \int^s_0e^{-\delta t}\left(\frac{1}{p}||\dot\sigma(t)||^p - (\beta +\alpha W_p(\rho,\sigma(t))^p) \right)dt\\
& \ge-\frac{1}{\delta}\left(\beta + 2^p\alpha W_p(\mu,\varrho)^p\right) + \\
&\quad \left(\frac{1}{p} - \alpha \left(\frac{2p}{\delta}\right)^p\right)\int^\infty_0e^{-\delta t}||\dot\sigma(t)||^pdt  \\
&\ge -\frac{1}{\delta}\left(\beta + 2^p\alpha W_p(\mu,\varrho)^p\right).
\end{align*}
The lower bound in \eqref{BoundsU} is now immediate.  Choosing $\sigma(t)=\mu$ for all $t$ in \eqref{GenValueFun} gives the upper bound in \eqref{BoundsU}.
\end{proof}
Next we derive the all important {\it Dynamic Programming Principle}. The proof here is not so different from well known arguments used to prove dynamic programming for \eqref{classVal} (see Lemma 7.1 of \cite{FS}), but we include it for
completeness. We also remark that the HJE \eqref{NewHJE} is an infinitesimal version of the dynamic programming principle \eqref{DPP}.
\begin{prop}\label{DPPprof}
For each $T\in [0,\infty)$ and $\mu\in\cMp$,
\begin{equation}\label{DPP}
U(\mu)=\inf\left\{e^{-\delta T}U(\sigma(T)) + \int^T_0e^{-\delta t}\left(\frac{1}{p}||\dot\sigma(t)||^p -\cV(\sigma(t)) \right)dt\right\}
\end{equation}
where the infimum is taken over paths  $\sigma\in AC_{p,\delta}(\cMp)$ with $\sigma(0)=\mu$.
\end{prop}

\begin{proof}
Let $\sigma$ be admissible for $\cU(\mu)$. Then $[0,\infty)\ni t\mapsto \sigma(t+T)$ is admissible for $\cU(\sigma(T))$ and
\begin{align*}
e^{-\delta T}\cU(\sigma(T)) +\int^T_0e^{-\delta t}\left(\frac{1}{p}||\dot\sigma(t)||^p -\cV(\sigma(t)) \right)dt &\le \int^\infty_0e^{-\delta (t+T)}\left(\frac{1}{p}||\dot\sigma(t+T)||^p -\cV(\sigma(t+T)) \right)dt  \\
&\quad + \int^T_0e^{-\delta t}\left(\frac{1}{p}||\dot\sigma(t)||^p -\cV(\sigma(t)) \right)dt \\
&=\int^\infty_0e^{-\delta t}\left(\frac{1}{p}||\dot\sigma(t)||^p -\cV(\sigma(t)) \right)dt.
\end{align*}
Consequently, the left hand of side of \eqref{DPP} is at least as large as the right hand side.

\par Again let $\sigma$ be admissible for $\cU(\mu)$ and $\epsilon>0$. Choose a path $\bar{\sigma}\in AC_{p,\delta}(\cMp)$ such that $s\mapsto \bar{\sigma}(s+T)$ is $\epsilon$ optimal for $\cU(\sigma(T))$. That is
$$
U(\sigma(T))+\epsilon >  \int^\infty_0e^{-\delta t}\left(\frac{1}{p}||\dot{\bar\sigma}(s+T)||^p -\cV(\bar\sigma(s+T)) \right)ds.
$$
Changing variables in the above integral $t=s+T$ and manipulating similar to how we did above gives
$$
e^{-\delta T}\cU(\sigma(T)) +\int^T_0e^{-\delta t}\left(\frac{1}{p}||\dot\sigma(t)||^p -\cV(\sigma(t)) \right)dt > -e^{-\delta T}\epsilon + U(\mu).
$$
It now follows that the right hand of side of \eqref{DPP} is no less than the left hand side.
\end{proof}

\begin{rem}
It is clear that the infimum in \eqref{DPP} can be taken over paths $\sigma\in AC_{p}([0,T],\cMp)$ with $\sigma(0)=\mu$.
\end{rem}
\begin{rem}\label{DPPminRemark}
Suppose that $\sigma$ is a minimizer for $\cU(\mu)$, then
\begin{align*}
\cU(\mu)&=\int^\infty_0e^{-\delta t}\left(\frac{1}{p}||\dot\sigma(t)||^p -\cV(\sigma(t)) \right)dt \\
&=\int^T_0e^{-\delta t}\left(\frac{1}{p}||\dot\sigma(t)||^p -\cV(\sigma(t)) \right)dt +\int^\infty_T e^{-\delta t}\left(\frac{1}{p}||\dot\sigma(t)||^p -\cV(\sigma(t)) \right)dt \\
&=\int^T_0e^{-\delta t}\left(\frac{1}{p}||\dot\sigma(t)||^p -\cV(\sigma(t)) \right)dt +e^{-\delta T}\int^\infty_0e^{-\delta s}\left(\frac{1}{p}||\dot{\sigma}(T+s)||^p -\cV(\sigma(T+s)) \right)dt \\
&\ge \int^T_0e^{-\delta t}\left(\frac{1}{p}||\dot\sigma(t)||^p -\cV(\sigma(t)) \right)dt + e^{-\delta T}\cU(\sigma(T))\\
&\ge \cU(\mu).
\end{align*}
The last inequality follows from \eqref{DPP}. Hence $\sigma$ is also a minimizer of
\begin{equation}\label{DPPfunctional}
AC_{p,\delta}(\cMp)\ni \rho\mapsto  e^{-\delta T}\cU(\rho(T))+\int^T_0e^{-\delta t}\left(\frac{1}{p}||\dot\rho(t)||^p -\cV(\rho(t)) \right)dt
\end{equation}
for each $T>0$.
\end{rem}
Next we derive continuity properties of value functions. First we show that $\cU$ is in general continuous and then refine this statement in terms of the modulus of
continuity of $\cV$.
\begin{prop}\label{ContinuityU}
$\cU\in C(\cMp)$.
\end{prop}

\begin{proof}
Assume $\mu_n\rightarrow \mu$ in $\cMp$ and let $\sigma$ be admissible for $\cU(\mu)$. Define for $\eta>0$
$$
\sigma_n(s):=
\begin{cases}
\underline{\sigma}_n(s/\eta), \quad 0\le s\le \eta\\
\sigma(s-\eta), \quad \eta\le s<\infty
\end{cases}
$$
where $\underline{\sigma}_n:[0,1]\rightarrow \cMp$ is a constant speed geodesic joining $\mu_n$ to $\mu$.  As $\sigma_n$ is admissible for $\cU(\mu_n)$
\begin{align*}
\cU(\mu_n) & \le \int^\infty_0e^{-\delta s}\left(\frac{1}{p}||\dot\sigma_n(s)||^p -\cV(\sigma_n(s)) \right)ds \\
& = \int^\eta_0 e^{-\delta s}\left(\frac{1}{p}||\dot\sigma_n(s)||^p -\cV(\sigma_n(s)) \right)ds  \\
&\quad + \int^\infty_\eta e^{-\delta s}\left(\frac{1}{p}||\dot\sigma_n(s)||^p -\cV(\sigma_n(s)) \right)ds \\
& = \int^\eta_0 e^{-\delta s}\left(\frac{1}{p}\left(\frac{W_p(\mu,\mu_n)}{\eta}\right)^p -\cV(\underline{\sigma}_n(s/\eta)) \right)ds\\
& \quad + \int^\infty_0 e^{-\delta (t+\eta)}\left(\frac{1}{p}||\dot\sigma(t)||^p -\cV(\sigma(t)) \right)dt.
\end{align*}
Since $W_p(\underline{\sigma}_n(s/\eta),\mu)\le C$ for some universal constant $C$, we can pass to the limit above using the continuity of $\cV$
to obtain
$$
\limsup_{n\rightarrow \infty}\cU(\mu_n)\le \frac{1-e^{-\delta \eta}}{\delta} \cV(\mu) + e^{-\delta \eta}\int^\infty_0e^{-\delta t}\left(\frac{1}{p}||\dot\sigma(t)||^p -\cV(\sigma(t)) \right)dt.
$$
As $\eta$ and $\sigma$ were arbitrary, we conclude that $\limsup_{n\rightarrow \infty}\cU(\mu_n)\le \cU(\mu)$.

\par Now choose any $\epsilon_n>0$ tending to $0$ as $n\rightarrow \infty$, and $\sigma_n$ admissible for $\cU(\mu_n)$ for which
$$
\cU(\mu_n)>-\epsilon_n + \int^\infty_0e^{-\delta s}\left(\frac{1}{p}||\dot\sigma_n(s)||^p -\cV(\sigma_n(s)) \right)ds.
$$
Define
$$
\sigma_{n,\eta}(t):=
\begin{cases}
\underline{\sigma}_n(1- t/\eta), \quad 0\le t\le \eta\\
\sigma_n(t-\eta), \quad \eta\le t<\infty
\end{cases}
$$
and observe that $\sigma_{n,\eta}$ is admissible for $\cU(\mu)$. Similar to our computations above,
 we find

 \begin{align*}
 \cU(\mu) &\le \int^\eta_0e^{-\delta t} \left\{\frac{1}{p}\left(\frac{W_p(\mu,\mu_n)}{\eta}\right)^p -\cV(\underline{\sigma}_n(1- t/\eta))\right\}dt \\
 & \quad + \int^\infty_0 e^{-\delta (\eta+s)}\left(\frac{1}{p}||\dot\sigma_n(s)||^p -\cV(\sigma_n(s)) \right)ds \\
&<\int^\eta_0e^{-\delta t} \left\{\frac{1}{p}\left(\frac{W_p(\mu,\mu_n)}{\eta}\right)^p -\cV(\underline{\sigma}_n(1- t/\eta))\right\}dt \\
&\quad + (\epsilon_n + \cU(\mu_n))e^{-\delta\eta}.
 \end{align*}
We easily compute $\cU(\mu)\le \liminf_{n\rightarrow\infty}\cU(\mu_n) e^{-\delta\eta} +\frac{1-e^{-\delta \eta}}{\delta} \cV(\mu)$. Since $\eta>0$ was arbitrary, $\cU$ is lower semicontinuous.
\end{proof}

\begin{rem}
It is worth noting the proof of continuity of the generalized value functions on infinite horizons requires less hypotheses than the proof for finite horizons (as the value function is continuous with no constraints on $\delta$).  See Proposition 2.5 of \cite{HyndKim}.
\end{rem}

\begin{prop}
If $V$ is uniformly continuous with modulus $\omega$, then $\cU$ is uniformly continuous with modulus $\omega/\delta$.
\end{prop}

\begin{proof}
Assume first that $\mu^1$ and $\mu^2$ are absolutely continuous with respect to Lebesgue measure. Let $T:\R^d\rightarrow \R^d$ be the unique Borel mapping
that pushes $\mu^1$ onto $\mu^2$ and
$$
W_p(\mu^1,\mu^2)^p=\int_{\R^d}|x-T(x)|^pd\mu^1(x)=\int_{\R^d}|y-T^{-1}(y)|^pd\mu^2(y).
$$
Let $\epsilon>0$ and choose $\sigma^2$ admissible for $\cU(\mu^2)$ such that
$$
\cU(\mu^2) > -\epsilon + \int^\infty_0e^{-\delta t}\left(\frac{1}{p}||\dot\sigma^2(t)||^p -\cV(\sigma^2(t)) \right)dt.
$$
By the probabilistic representation of $AC_{p,\delta}(\cMp)$ paths detailed in Theorem \ref{ProbRep}, there is a Borel probability measure $\eta$ on $\R^d \times \Gamma$ such that
$$
\sigma^2(t):=e(t)_{\#}\eta, \quad t\ge 0.
$$
Here $\Gamma$ is the space $C([0,\infty); \R^d)$ with the topology of local uniform convergence, 
\begin{equation}\label{EvalMap}
e(t): \R^d\times \Gamma\rightarrow \R^d; (x,\gamma)\mapsto \gamma(t)
\end{equation}
is the {\it evaluation map} and $\eta$ is concentrated on $(x,\gamma)$ satisfying $\gamma(0)=x$ and $\dot\gamma(t)=v(\gamma(t),t)$, a.e. $t\ge 0$. The field $v$ is the minimal velocity of $\sigma$.

\par Define a map
$$
S(x,\gamma):=(T^{-1}(x), \gamma + T^{-1}(x)-x)
$$
and a measure on $\R^d\times \Gamma$
$$
\lambda:=S_{\#}\eta.
$$
Also set
$$
\sigma^1(t):=e(t)_{\#}\lambda, \quad t\ge 0.
$$
As in Proposition 2.6 of \cite{HyndKim}, one checks that $\sigma^1$ is admissible for $\cU(\mu^2)$, that $||\dot\sigma^1(t)||\le ||\dot\sigma^2(t)||$ for almost every $t\ge 0$, and
$$
W_p(\sigma^1(t),\sigma^2(t))\le W_p(\mu^1, \mu^2).
$$
Therefore, it follows that
\begin{align*}
\cU(\mu^1)-\cU(\mu^2)&\le \epsilon +  \int^\infty_0e^{-\delta t}\left(\frac{1}{p}||\dot\sigma^1(t)||^p -\frac{1}{p}||\dot\sigma^2(t)||^p+ \cV(\sigma^2(t))- \cV(\sigma^1(t)) \right)dt \\
&\le \epsilon + \int^\infty_0 e^{-\delta t}\omega(W_p(\sigma^1(t),\sigma^2(t)))dt\\
& \le \epsilon +\int^\infty_0 e^{-\delta t}\omega(W_p(\mu^1,\mu^2))dt\\
&= \epsilon +\frac{1}{\delta} \omega(W_p(\mu^1,\mu^2)).
\end{align*}
Interchanging, $\mu^1$ and $\mu^2$ yields the claim in the case where $\mu^1$ and $\mu^2$ are absolutely continuous. The general assertion now follows from the density of
absolutely continuous measures in $\cMp$ (Lemma 7.1.10 \cite{AGS}) and the previous proposition.
\end{proof}



\section{Minimizing trajectories}\label{MinSec}
In this section, we prove our main existence result, Theorem \ref{WeakExistenceThm}. Our main tool is a compactness lemma, which involves the weak or
narrow convergence of sequences of paths in $AC_{p,\delta}(\cMp)$.  We begin our study of minimizing paths by deriving
 the Euler-Poisson system \eqref{StaticEP} as a necessary condition.

\begin{thm}\label{EPthm}
Assume $\cV$ is Lipschitz continuous, and that for each $\mu\in \cMp$, there is a mapping $\nabla\cV(\mu)\in L^p(\mu)$ for which
\begin{equation}\label{DerivativeV}
\lim_{\epsilon\rightarrow 0}\frac{\cV((\id +\epsilon \eta)_{\#}\mu) - \cV(\mu)}{\epsilon} = \int_{\R^d}\nabla \cV(\mu)\cdot \eta d\mu
\end{equation}
for each $\eta \in C^\infty_c(\R^d; \R^d)$.  Then for any minimizer $\sigma$ of $\cU(\mu)$, the equations
$$
\begin{cases}
\partial_t(\sigma |v|^{p-2}v) + \nabla\cdot (\sigma|v|^{p-2} v\otimes v) = -\sigma \nabla\cV(\sigma)  + \sigma\delta |v|^{p-2}v \\
\hspace{1.28in}\partial_t\sigma +\nabla\cdot(\sigma v)=0
\end{cases}
$$
hold in the sense of distributions on $\R^d\times(0,\infty)$; here $v$ is the minimal velocity for $\sigma$.
\end{thm}
\begin{proof}  Let $\Psi\in C^\infty_c(\R^d\times (0,\infty); \R^d)$, and set
$$
\begin{cases}
\sigma^\epsilon(t):=(\id +\epsilon \Psi(t))_{\#}\sigma(t)\\
v^\epsilon(t):=(v(t) + \epsilon(\partial_t\Psi(t) + (v(t)\cdot \nabla)\Psi(t)))\circ (\id +\epsilon\Psi(t))^{-1}
\end{cases}
$$
for $\epsilon $ so small that $x\mapsto x+\epsilon\Psi(x,t)$ is invertible for each $t>0$. One checks that
$$
\partial_t\sigma^\epsilon + \nabla\cdot (\sigma^\epsilon v^\epsilon)=0, \quad \R^d\times (0,\infty)
$$
and that $\sigma^\epsilon$ is admissible for $\cU(\mu)$. Thus
\begin{align*}
\cU(\mu) & =\int^\I_0e^{-\delta t}\left(\frac{1}{p}||\dot\sigma(t)||^p -\cV(\sigma(t)) \right)dt\\
& =\int^\I_0e^{-\delta t}\left(\frac{1}{p}\int_{\R^d}|v(x,t)|^pd\sigma_t(x) -\cV(\sigma(t)) \right)dt\\
&\le \int^\I_0e^{-\delta t}\left(\frac{1}{p}||\dot\sigma^\epsilon(t)||^p -\cV(\sigma^\epsilon(t)) \right)dt\\
&\le \int^\I_0e^{-\delta t}\left(\frac{1}{p}\int_{\R^d}|v^\epsilon(x,t)|^pd\sigma^\epsilon_t(x) -\cV(\sigma^\epsilon(t)) \right)dt\\
&= \int^\I_0e^{-\delta t}\left(\frac{1}{p}\int_{\R^d}|v + \epsilon(\partial_t \Psi+ (v\cdot \nabla)\Psi)|^pd\sigma_t(x) -\cV((\id +\epsilon \Psi(t))_{\#}\sigma(t)) \right)dt.
\end{align*}
Consequently, the derivative on the right hand side expression above must be zero when taken at $\epsilon=0$. Employing \eqref{DerivativeV} and our
assumption that $\cV$ is Lipschitz, we use standard limit theorems from measure theory to compute the derivative in question and find
\begin{equation}\label{StaticEPIntegral2}
0= \int^\I_0 \int_{\R^d}e^{-\delta t}\left\{(\partial_t\Psi+ (v\cdot \nabla)\Psi)\cdot |v|^{p-2}v - \nabla\cV(\sigma)\cdot \Psi\right\}d\sigma_t(x)dt.
\end{equation}
Choosing $\Psi(x,t)=e^{\delta t}\Phi(x,t)$, where $\Phi\in C^\infty_c(\R^d\times (0,\infty); \R^d)$ gives \eqref{StaticEP}.
\end{proof}
Note that \eqref{StaticEPIntegral2} also holds for $\Psi\in C^1_c(\R^d\times(0,\infty);\R^d)$. Also observe that by selecting $\Psi(x,t)=\eta(x)f(t)$ in \eqref{StaticEPIntegral2}, where $f\in C^\infty_c(0,\infty)$ and $\eta\in C^1_c(\R^d)$, gives
\begin{equation}\label{StaticEPIntegral}
\frac{d}{dt}\int_{\R^d}e^{-\delta t}|v|^{p-2}v\cdot \eta d\sigma_t = \int_{\R^d}e^{-\delta t}\left\{|v|^{p-2}v\cdot((v\cdot\nabla) \eta) - \nabla \mathcal{V}(\sigma)\cdot \eta\right\}d\sigma_t.
\end{equation}
Equation \eqref{StaticEPIntegral} holds in the sense of distributions on $(0,\infty)$.

\par It now follows that the $L^1[0,\infty)$ function
$$
[0,\infty)\ni t\mapsto \int_{\R^d}e^{-\delta t}|v(x,t)|^{p-2}v(x,t)\cdot \eta(x) d\sigma_t(x)
$$
can be identified with an absolutely continuous function belonging to $W^{1,1}[0,\infty)$. Moreover, standard arguments can be used to prove $[0,\infty)\ni t\mapsto \sigma(t)|v(t)|^{p-2}v(t)$ has a continuous representative with values in the dual space of the closure of $C^1_c(\R^d; \R^d)$ with the norm
$$
||\eta||:=\sup_{\R^d}|\eta| + \sup_{\R^d}|\nabla\eta|.
$$
See the proof of Lemma 8.1.2 in \cite{AGS} for more details on this technical point.

\begin{cor}
Assume the hypotheses of Theorem \ref{EPthm} and identify $t\mapsto \sigma(t)|v(t)|^{p-2}v(t)\in \left(C^1_c(\R^d; \R^d)\right)'$ with its continuous representative. Then for all $\eta \in C^1_c(\R^d; \R^d)$
\begin{equation}\label{LimitTinf}
\lim_{t \rightarrow \infty}e^{-\delta t}\int_{\R^d}|v(t,x)|^{p-2}v(t,x)\cdot \eta(x)d\sigma_t(x)=0
\end{equation}
and
\begin{equation}\label{WeakConj}
\liminf_{\epsilon\rightarrow 0^+}\frac{\cU((\id +\epsilon \eta)_{\#}\sigma(t))-\cU(\sigma(t))}{\epsilon}\ge -\int_{\R^d}|v(x,t)|^{p-2}v(x,t)\cdot \eta(x) d\sigma_t(x)
\end{equation}
for $t>0$.
\end{cor}

\begin{proof}  It also not difficult to see that if  $h\in C^\infty([0,\infty))$, $h(0)=0$, and $h(t)=1$ for all $t$ large, then $\Psi(t,x)=h(t)\eta(x)$ is a valid test function in \eqref{StaticEPIntegral2}.  Substituting this test function yields
\begin{align*}
0 &= \int^\I_0 \int_{\R^d}e^{-\delta t}\left\{(\partial_t\Psi+ (v\cdot \nabla)\Psi)\cdot |v|^{p-2}v - \nabla\cV(\sigma(t))\cdot \Psi(t)\right\}d\sigma_t(x)dt \\
 &= \lim_{T\rightarrow \infty}\int^T_0 \int_{\R^d}e^{-\delta t}\left\{(\partial_t\Psi+ (v\cdot \nabla)\Psi)\cdot |v|^{p-2}v - \nabla\cV(\sigma(t))\cdot \Psi(t)\right\}d\sigma_t(x)dt \\
  & = \lim_{T\rightarrow \infty}\int^T_0 \int_{\R^d}e^{-\delta t}\left\{(g'(t)\eta+ (v\cdot \nabla)g(t)\eta)\cdot |v|^{p-2}v - \nabla\cV(\sigma(t))\cdot g(t)\eta\right\}d\sigma_t(x)dt \\
  & = \lim_{T\rightarrow \infty} \left\{g(t)\left.\int_{\R^d}e^{-\delta t}|v|^{p-2}v\cdot \eta d\sigma_t\right|^T_0 +\right.  \\
& \left. \int^T_0 g(t)\left[-\frac{d}{dt}\int_{\R^d}e^{-\delta t}|v|^{p-2}v\cdot \eta d\sigma_t +\int_{\R^d}e^{-\delta t}\left\{|v|^{p-2}v\cdot((v\cdot\nabla) \eta) - \nabla V(\sigma)\cdot \eta\right\}d\sigma_t \right] \right\}\\
&=\lim_{T\rightarrow \infty}\int_{\R^d}e^{-\delta T}|v(x,T)|^{p-2}v(x,T)\cdot \eta(x) d\sigma_T(x)
\end{align*}
which is \eqref{LimitTinf}.

\par Now let  $f\in C^\infty_c(0,\infty)$ with $f(T)=1$, and set
$$
\Psi(x,t):=\eta(x)f(t)\in C^\infty_c(\R^d\times (0,\infty);\R^d).
$$
Also denote $\sigma^\epsilon$ and $v^\epsilon$ as in the proof of Theorem \ref{EPthm}.
As $\sigma$ is a minimizer of \eqref{DPPfunctional},
\begin{align*}
\cU(\mu)&=e^{-\delta T}\cU(\sigma(T))+\int^T_0e^{-\delta t}\left(\frac{1}{p}||\dot\sigma(t)||^p -\cV(\sigma(t)) \right)dt\\
&=e^{-\delta T}\cU(\sigma(s))+\int^T_0e^{-\delta t}\left(\frac{1}{p}\int_{\R^d}|v(x,t)|^pd\sigma_t(x) -\cV(\sigma(t)) \right)dt\\
&\le e^{-\delta T}\cU(\sigma^\epsilon(T))+\int^T_0e^{-\delta t}\left(\frac{1}{p}||\dot\sigma^\epsilon(t)||^p -\cV(\sigma^\epsilon(t)) \right)dt\\
&\le e^{-\delta T}\cU(\sigma^\epsilon(T))+\int^T_0e^{-\delta t}\left(\frac{1}{p}\int_{\R^d}|v^\epsilon(x,t)|^pd\sigma^\epsilon_t(x)
-\cV(\sigma^\epsilon(t)) \right)dt\\
&\le e^{-\delta T}\cU((\id+\epsilon\eta)_{\#}\sigma(T))\\
& \quad +\int^T_0e^{-\delta t}\left(\frac{1}{p}\int_{\R^d}|v+\epsilon(\partial_t\Psi + v\cdot \nabla\Psi)|^pd\sigma_t -\cV((\id+\epsilon\Psi(t))_{\#}\sigma(t)) \right)dt\\
\end{align*}
for each $\epsilon>0$ small enough. Similar to computations performed in the proof of Theorem \ref{EPthm}
\begin{align*}
\liminf_{\epsilon\rightarrow 0^+}\frac{\cU((\id +\epsilon \eta)_{\#}\sigma(T))-\cU(\sigma(T))}{\epsilon}  & \ge -e^{\delta T}\int^T_0 \int_{\R^d}e^{-\delta t}\left\{(\partial_t\Psi+ (v\cdot \nabla)\Psi)\cdot |v|^{p-2}v  \right. \\
& \hspace{1.5in} \left.- \nabla\cV(\sigma(t))\cdot \Psi(t)\right\}d\sigma_tdt\\
 & = -e^{\delta T} \int^T_0\left[ f'(t) \left\{\int_{\R^d}e^{-\delta t}|v|^{p-2}v\cdot \eta d\sigma_t\right\} +  \right. \\
 &\quad \left. f(t)\left\{  \int_{\R^d}e^{-\delta t}\left\{|v|^{p-2}v\cdot((v\cdot\nabla) \eta) - \nabla \mathcal{V}(\sigma)\cdot \eta\right\}d\sigma_t \right\}\right]dt \\
 &=  -\int_{\R^d}|v(x,T)|^{p-2}v(x,T)\cdot \eta(x) d\sigma_T(x)
\end{align*}
The last inequality follows from an integration by parts and the use of the identity \eqref{StaticEPIntegral}.
\end{proof}
We believe that more is true.  We prove a special case of the below conjecture in the case of simple potentials; see Corollary \ref{SpecConj}.
\begin{conj}\label{GradConj} Under the assumptions of Theorem \ref{EPthm},
$$
-|v(t)|^{p-2}v(t)\in \nabla^- \cU(\sigma(t)), \quad \text{a.e}\;\; t\in (0,\infty).
$$
$\nabla^-\mathcal{U}$ is specified in Definition \ref{DiffDef}.
\end{conj}

\par We now will address some issues related to the existence of minimizing paths by using ideas from the calculus of variations. It will be necessary for us to employ the narrow convergence of
measures. Recall that the L\'{e}vy-Prokhorov metric
\begin{equation}\label{LPmetric}
\Lambda(\mu,\nu):=\inf\left\{\epsilon>0: \mu(A)\le \nu(A^\epsilon)+\epsilon, \; \nu(A)\le \mu(A^\epsilon)+\epsilon, \; \text{for Borel}\; A\subset\R^d\right\}
\end{equation}
completely metrizes $\cPp$ (see Chapter 6 of \cite{Bill}); here $A_\epsilon:=\cup_{z\in A}B_{\epsilon}(z)$. Moreover, the following inequality
\begin{equation}\label{deltaIneq}
\Lambda^2\le W_1
\end{equation}
holds (Corollary 2.18 of \cite{Huber}). The inequality \eqref{deltaIneq} is critical in the following compactness lemma.
\begin{lem}\label{BlowupLem}
Assume $\{\sigma^k\}_{k\in \N}\subset AC_{p,\delta}(\cMp)$ for some $p\in (1,\infty)$ and $\delta>0$. Further suppose
$$
\sigma^k(0)=\mu
$$
and
$$
\sup_{k}\int^\infty_0e^{-\delta t}||\dot\sigma^k(t)||^pdt<\infty.
$$
Then there is a subsequence $\{\sigma^{k_j}\}_{j\in \N}$ and $\sigma\in  AC_{p,\delta}(\cMp)$ such that\\
(i) $\sigma^{k_j}\rightarrow \sigma$ locally uniformly in $(\cPp,\Lambda)$, \\
(ii)
$$
\liminf_{j\rightarrow\infty}\int^\infty_0e^{-\delta t}||\dot\sigma^{k_j}(t)||^pdt\ge \int^\infty_0e^{-\delta t}||\dot\sigma(t)||^pdt,
$$
and $(iii)$
$$
\lim_{j\rightarrow \infty} \int^\infty_0 e^{-\delta t}\Lambda(\sigma^{k_j}(t),\sigma(t))^{2r}dt=0
$$
for all $1\le r<p$.
\end{lem}
\begin{proof}
Note for each $T>0$,
$$
\int^T_0 ||\dot\sigma^k(t)||^pdt \le e^{\delta T}\int^T_0 e^{-\delta t}||\dot\sigma^k(t)||^pdt \le  e^{\delta T}\int^\infty_0 e^{-\delta t}||\dot\sigma^k(t)||^pdt\le Ce^{\delta T}.
$$
By Proposition 4.1 of our previous work \cite{HyndKim}, there is a subsequence of $\sigma^k$ convergent in $C([0,T]; (\cPp,\Lambda))$. Using a routine diagonalization argument, we obtain
a sequence $\{\sigma^{k_j}\}_{j\in \N}$ and $\sigma \in C([0,\infty); (\cPp,\Lambda))$ such that $\sigma=\lim_{j\rightarrow \infty}\sigma^{k_j}$ locally uniformly in the narrow topology.

\par Also note $t\mapsto ||\dot\sigma^{k_j}(t)||$ is bounded in $L^p((0,\infty); e^{-\delta t}dt)$ and so has a further subsequence (not relabeled here) that converges weakly to some $g$.
By the lower semicontinuity properties of $W_p$
\begin{align*}
W_p(\sigma(t_1), \sigma(t_2)) & \le \liminf_{j\rightarrow \infty}W_p(\sigma^{k_j}(t_1), \sigma^{k_j}(t_2)) \\
&\le  \liminf_{j\rightarrow \infty}\int^{t_2}_{t_1}||\dot\sigma^{k_j}(t)||dt\\
&=  \liminf_{j\rightarrow \infty}\int^{\infty}_{0}e^{-\delta t}\left(e^{\delta t}\chi_{[t_1,t_2]}(t)\right)||\dot\sigma^{k_j}(t)||dt\\
&= \int^{\infty}_{0}e^{-\delta t}\left(e^{\delta t}\chi_{[t_1,t_2]}(t)\right)g(t)dt\\
&= \int^{t_2}_{t_1}g(t)dt\\
\end{align*}
for $0\le t_1\le t_2<\infty$. As $g\in L^p_{\text{loc}}(0,\infty)$, $||\dot\sigma(t)||\le g(t)$ for almost every $t\ge 0$. Moreover, weak convergence implies
$$
\int^\infty_0 e^{-\delta t}||\dot\sigma(t)||dt\le \int^\infty_0 e^{-\delta t}||g(t)||dt \le \liminf_{j\rightarrow \infty}\int^\infty_0 e^{-\delta t}||\dot\sigma^{k_j}(t)||dt  .
$$
\par By Egorov's theorem, for every $\epsilon>0$, there is a Borel measurable $A_\epsilon\subset [0,\infty)$ such that $\int_{\A_\epsilon}e^{-\delta t}dt \le \epsilon$ and $\sigma^{k_j}\rightarrow \sigma$ on $[0,\infty)\setminus A_\epsilon$ uniformly in the narrow topology. Consequently for $1\le r<p$,

\begin{align*}
\int^\I_0e^{-\delta t}\Lambda(\sigma^{k_j}(t), \sigma(t))^{2r}dt & = \int_{[0,\infty)\setminus A_\epsilon}e^{-\delta t}\Lambda(\sigma^{k_j}(t), \sigma(t))^{2r}dt + \int_{A_\epsilon}e^{-\delta t}\Lambda(\sigma^{k_j}(t), \sigma(t))^{2r}dt \\
& \le \frac{1}{\delta}\left[\sup_{t\in [0,\infty)\setminus A_\epsilon}\Lambda(\sigma^{k_j}(t),\sigma(t))\right]^{2r} + \left(\int^\I_0e^{-\delta t}\Lambda(\sigma^{k_j}(t), \sigma(t))^{2p}dt\right)^{r/p}\epsilon^{1-r/p}\\
& \le \frac{1}{\delta}\left[\sup_{t\in [0,\infty)\setminus A_\epsilon}\Lambda(\sigma^{k_j}(t),\sigma(t))\right]^{2r} + \left(\int^\I_0e^{-\delta t}W_p(\sigma^{k_j}(t), \sigma(t))^{p}dt\right)^{r/p}\epsilon^{1-r/p}\\
& \le \frac{1}{\delta}\left[\sup_{t\in [0,\infty)\setminus A_\epsilon}\Lambda(\sigma^{k_j}(t),\sigma(t))\right]^{2r} +C\epsilon^{1-r/p}.
\end{align*}
The final estimate follows from the weighted Poincare inequality and the hypotheses of this theorem. Thus,
$$
\limsup_{j\rightarrow \infty}\int^\I_0e^{-\delta t}\Lambda(\sigma^{k_j}(t), \sigma(t))^{2r}dt \le C\epsilon^{1-r/p}.
$$
The claim now follows from sending $\epsilon \rightarrow 0^+$.
\end{proof}
We are finally in position to prove Theorem \ref{WeakExistenceThm}.
\begin{proof} (of Theorem \ref{WeakExistenceThm}) Let $\mu\in \cMp$ and $\epsilon_k$ be a sequence of positive numbers tending to 0 as $k\rightarrow \infty$. Choose paths $\sigma^k$ admissible for $\cU(\mu)$  such that
\begin{equation}\label{ExistenceINeq}
\cU(\mu)> -\epsilon_k + \int^\infty_0 e^{-\delta t }\left(\frac{1}{p}||\dot\sigma^{k}(t)||^p -\cV(\sigma^k(t))\right)dt
\end{equation}
for $k\in \N$. By assumption \eqref{cond18}, we manipulate \eqref{ExistenceINeq} as in Lemma \ref{BoundsLem} and conclude $\int^\infty_0 e^{-\delta t }||\dot\sigma^{k}(t)||^pdt\le C$. As a result, the sequence
$\{\sigma^k\}_{k\in \N}$ satisfies the hypotheses of Lemma \ref{BlowupLem}. Hence, there is a subsequence $\{\sigma^{k_j}\}_{j\in \N}$ and $\sigma$ admissible for $\cU(\mu)$ for which $\sigma^{k_j}$ converges to $\sigma$ as described
in the previous assertion. Employing assumption \eqref{weakVBound}, we apply the dominated convergence theorem to find
$$
\lim_{j\rightarrow \infty}\int^\infty_0 \cV(\sigma^{k_j}(t))e^{-\delta t}dt =\int^\infty_0 \cV(\sigma(t))e^{-\delta t}dt.
$$
The claim is now immediate from passing to the limit $k=k_j\rightarrow \infty$ in \eqref{ExistenceINeq}.
\end{proof}

\section{Simple potentials}\label{SimpleSec}
In this section, we will focus on value functions in the case of simple potentials \eqref{SimpleV}
$$
\cV(\mu)=\int_{\R^d}V(x)d\mu(x).
$$
We shall further assume $V\in C^1(\R^d)$ and satisfies
\begin{equation}\label{littleVbounds}
|V(x)|\le a|x|^r+b
\end{equation}
for some $a, b\in \R$ and $1\le r<p$. Under this assumption, classical value functions $u=u(x)$ \eqref{classVal} can be shown to well defined, continuous
and have minimizing paths for each $x\in \R^d$ (see Lemma \ref{ClassCompactness} in the appendix).  Moreover, using the compactness built in to
this classical optimization problem, we obtain a measurable flow map associated with minimizing paths.
This will be crucial to our proof of Proposition \ref{SpecForm}.

\begin{prop}\label{MeasSelect} Assume \eqref{littleVbounds} and define the set valued mapping
$$
F(x):=\left\{\gamma\in AC_{p,\delta}(\R^d): u(x)=\int^\I_0e^{-\delta t}\left(\frac{1}{p}|\dot\gamma(t)|^p -V(\gamma(t)) \right)dt, \;\gamma(0)=x\right\}
$$
for $x\in \R^d$.  Then there is a measurable map $\Phi:\R^d\rightarrow AC_{p,\delta}(\R^d)$ such that for each $x\in \R^d$, $\Phi(x)\in F(x).$
\end{prop}

\begin{proof}
For $\gamma,\xi\in AC_{p,\delta}(\R^d)$ define
$$
\Sigma(\gamma,\xi):=\Pi(\gamma,\xi)+\left[\int^\infty_0e^{-\delta t}|\dot\gamma(t)-\dot\xi(t)|^pdt\right]^{1/p}
$$
where
\begin{equation}\label{PiMetricDef}
\Pi(\gamma,\xi)=\sum^\infty_{k=0}\frac{1}{2^k}\frac{\max_{0\le t\le k}|\gamma(t)-\xi(t)|}{1+\max_{0\le t\le k}|\gamma(t)-\xi(t)|}.
\end{equation}
In Proposition \ref{PiMetric} of the appendix, we verify that $\Pi$ makes $C([0,\infty);\R^d)$ into a complete, separable metric space. Moreover, convergence under $\Pi$ is equivalent to local uniform convergence of $\R^d$ valued paths on $[0,\infty)$.  Employing these facts about $\Pi$, it is straightforward to check
that $AC_{p,\delta}(\R^d)$ is a complete, separable metric space under the distance $\Sigma$.

\par According to Lemma \ref{ClassCompactness} in the appendix, $F(x)\neq \emptyset$ for each $x\in \R^d$. It is also routine to verify that $F(x)\subset AC_{p,\delta}(\R^d)$ is closed.  By Theorem 8.3.1 of \cite{AB}, it suffices to show that for each $\eta\in AC_{p,\delta}(\R^d)$,
$$
x\mapsto \text{dist}(\eta,F(x)) \; \text{is Borel measurable}
$$
to conclude the assertion ($\text{dist}(\eta,S):=\inf_{\xi\in S}\Sigma(\eta,\xi)$, $S\subset AC_{p,\delta}(\R^d)$). We show in fact that this function is lower-semicontinuous on $\R^d$.

\par To this end, assume $x_n\rightarrow x\in \R^d$ and choose $x_{n_j}$ such that
$$
\liminf_{n\rightarrow \infty}\text{dist}(\eta,F(x_n))=\lim_{j\rightarrow \infty}\text{dist}(\eta,F(x_{n_j}))
$$
Employing Lemma \ref{ClassCompactness}, we may select  $\gamma_j\in F(x_{n_j})$ such that $\text{dist}(\eta,F(x_{n_j}))=\Sigma(\eta,\gamma_j)$. Notice
$$
\int^\I_0e^{-\delta t}\left(\frac{1}{p}|\dot\gamma_j(t)|^p -V(\gamma_j(t)) \right)dt=u(x_{n_j})\le C
$$
as $u\in C(\R^d)$. It follows now from \eqref{littleVbounds} and the weighted Poincar\'{e} inequality that $\int^\I_0e^{-\delta t}|\dot\gamma_j(t)|^pdt$ is bounded independently of $j\in \N$.  Moreover,
 $|\gamma_j(0)|=|x_j|\le C$. By Lemma \ref{ClassCompactness}, there is a
$\gamma\in AC_{p,\delta}(\R^d)$ such that a subsequence of $\{\gamma_j\}_{j\in \N}$ (not relabeled) converges $\gamma$  in
$(C([0,\infty);\R^d),\Pi)$ and $\dot\gamma_j\rightarrow\dot\gamma$ weakly in $L^p([0,\infty);e^{-\delta t}dt)$.  In particular, $\gamma\in F(x)$.  This convergence implies
\begin{align*}
\text{dist}(\eta,F(x)) & \le \Sigma(\eta,\gamma)\\
&\le \liminf_{j\rightarrow \infty}\Sigma(\eta,\gamma_j)\\
&= \liminf_{j\rightarrow \infty}\text{dist}(\eta, F(x_{n_j}))\\
&\le \liminf_{n\rightarrow \infty}\text{dist}(\eta, F(x_{n})).
\end{align*}
\end{proof}
Proposition \ref{MeasSelect} establishes that for each $x$, $\Phi(x)$ is a minimizing
path for $u(x)$. Define a new map
\begin{equation}\label{MeasFlow}
\Psi:\R^d\times [0,\infty)\rightarrow \R^d; (x,t)\mapsto e(t)\circ(x,\Phi(x))
\end{equation}
which is measurable, since it is the composition of measurable mappings (recall the evaluation map $e(t)$ defined in \eqref{EvalMap}).  And by definition, $t\mapsto\Psi(x,t)$ satisfies the optimality equations \eqref{ClassicalGradFlow}
$$
\begin{cases}
|\partial_t\Psi(x,t)|^{p-2}\partial_t\Psi(x,t)=-\nabla u(\Psi(x,t)), \quad t>0\\
\Psi(x,0)=x
\end{cases}.
$$
Thus $\Psi$ is a measurable flow map associated with the ODE $|\dot\gamma|^{p-2}\dot\gamma=-\nabla u(\gamma)$. Also note that since the paths $t\mapsto \Psi(x,t)$ satisfy the Euler-Lagrange equations \eqref{EulLagDel}, they are also $C^1$.  We shall now exploit this map to deduce the formula \eqref{SpecValueFun}.

\begin{proof} (of Proposition \ref{SpecForm}) 1. Let $\sigma$ be an admissible path for $\cU(\mu)$ and employ Theorem \ref{ProbRep} to write
$$
\sigma(t)=e(t)_\#\eta, \quad t\ge 0.
$$
Note that
\begin{align*}
\int^\infty_0e^{-\delta t}\left(\frac{1}{p}||\dot\sigma(t)||^p -\cV(\sigma(t)) \right)dt & = \int^\infty_0e^{-\delta t}\left(\frac{1}{p}\int_{\R^d}|v(x,t)|^pd\sigma_t(x) - \int_{\R^d}V(x)d\sigma_t(x) \right)dt \\
&=\int^\infty_0e^{-\delta t}\int_{\R^d}\left\{\frac{1}{p}|v(x,t)|^p - V(x)\right\}d\sigma_t(x)dt \\
&=\int^\infty_0e^{-\delta t}\int_{\R^d\times \Gamma}\left\{\frac{1}{p}|v(\gamma(t),t)|^p - V(\gamma(t))\right\}d\eta(x,\gamma)dt \\
&=\int_{\R^d\times \Gamma}\int^\infty_0e^{-\delta t}\left\{\frac{1}{p}|v(\gamma(t),t)|^p - V(\gamma(t))\right\}dt d\eta(x,\gamma)\\
&=\int_{\R^d\times \Gamma}\int^\infty_0e^{-\delta t}\left\{\frac{1}{p}|\dot\gamma(t)|^p - V(\gamma(t))\right\}dt d\eta(x,\gamma)\\
&\ge \int_{\R^d\times \Gamma}u(x)d\eta(x,\gamma)\\
&=\int_{\R^d}u(x)d\mu(x).
\end{align*}
The interchange of integrals follows from Remark \ref{FubiniRemark}, assumption \eqref{littleVbounds} along with \eqref{1DPoincare}, and a routine application of Fubini's theorem.
We leave the details to the reader.

\par 2. Now define the path  $\sigma(t):=\Psi(t)_\#\mu$ for $t\ge 0$, where $\Psi$ is defined in \eqref{MeasFlow}. Since $t\mapsto \Psi(x,t)$ is a minimizer for $u(x)$, \eqref{cond18} implies
$$
\int_{\R^d}\int^\infty_0 e^{-\delta t}|\partial_t\Psi(t,x)|^pdt d\mu(x)<\infty
$$
for $\mu\in \cMp$; this bound follows closely to the proof of Lemma \ref{BoundsLem}.   Also notice for $0\le s< t<\infty $,
\begin{align*}
W_p(\sigma(t),\sigma(s))^p & \le \int_{\R^d}|\Psi(t,x)-\Psi(s,x)|^pd\mu(x) \\
& \le \int_{\R^d}\left(\int^t_s|\partial_t\Psi(\tau,x)|d\tau\right)^p d\mu(x)\\
& \le (t-s)^{p-1}\int_{\R^d}\int^t_s|\partial_t\Psi(\tau,x)|^pd\tau d\mu(x)
\end{align*}
which implies
$$
\left(\frac{W_p(\sigma(t),\sigma(s))}{t-s}\right)^p\le \frac{1}{t-s}\int^t_s \left(\int_{\R^d}|\partial_t\Psi(\tau,x)|^pd\mu(x)\right)d\tau.
$$
As a result,
$$
||\dot\sigma(t)||^p\le \int_{\R^d}|\partial_t\Psi(t,x)|^pd\mu(x)
$$
for a.e. $t\ge 0$.  Therefore,

\begin{align*}
\cU(\mu)&\le \int^\infty_0e^{-\delta t}\left(\frac{1}{p}||\dot\sigma(t)||^p -\cV(\sigma(t)) \right)dt\\
&\le  \int^\infty_0 e^{-\delta t}\left( \int_{\R^d}\left(\frac{1}{p}|\partial_t\Psi(t,x)|^p - V(\Psi(t,x))\right)d\mu(x)\right)dt\\
&\le \int_{\R^d}\left( \int^\infty_0 e^{-\delta t}\left(\frac{1}{p}|\partial_t\Psi(t,x)|^p - V(\Psi(t,x))\right)dt\right)d\mu(x)\\
&=\int_{\R^d}u(x)d\mu(x).
\end{align*}
In particular, $\sigma$ is optimal for $\cU(\mu)$.
\end{proof}

\begin{cor}\label{CompareEqLemma}  Let $\sigma$ be a minimizing path for $\cU(\mu)$ with minimal velocity $v$.  Then
\begin{equation}\label{gradcondae}
|v(x,t)|^{p-2}v(x,t)=-\nabla u(x), \quad \sigma(t)\;\; \text{a.e.} \; x\in \R^d
\end{equation}
for Lebesgue almost every $t>0$.
\end{cor}

\begin{proof}  From part 1 of the proof of Proposition \ref{SpecForm}, we see
that any minimizing path's $t\mapsto\sigma(t)=e(t)_\#\eta$ is
concentrated on $(x,\gamma)$ where $\gamma$ is a minimizer for
$u(x)$. By \eqref{ClassicalGradFlow}, we conclude for every $t>0$ and $\eta$ almost every $(x,\gamma)$
$$
|\dot\gamma(t)|^{p-2}\dot\gamma(t)=-\nabla u(\gamma(t)).
$$
From the canonical uniqueness of minimal velocities (Proposition 8.4.5 of \cite{AGS}), we have for Lebesgue almost every $t>0$ and
$\eta$ almost every $(x,\gamma)$
$$
|v(\gamma(t),t)|^{p-2}v(\gamma(t),t)=|\dot\gamma(t)|^{p-2}\dot\gamma(t).
$$
In particular, we have for Lebesgue almost every $t>0$
$$
|v(e_t(x,\gamma),t)|^{p-2}v(e_t(x,\gamma),t)=-\nabla u(e_t(x,\gamma))
$$
for $\eta$ almost every $(x,\gamma)$.
Again by the probabilistic representation $\sigma(t)=e(t)_\#\eta$ from which we conclude \eqref{gradcondae}.
\end{proof}

\begin{ex}
Assume $V(x) = w\cdot x + c$. Here $w\in \R^d$ and $c\in \R$ are fixed. The associated classical value function is
$$
u(x)=-\frac{1}{\delta}\left(\frac{|w|^q}{q} + w\cdot x +c\right), \quad x\in \R^d
$$
with flow map $\Psi(x,t)=x+ t\left|\frac{w}{\delta}\right|^{q-2}\frac{w}{\delta}$.  By Proposition  \ref{SpecForm}
$$
\cU(\mu)=-\frac{1}{\delta}\left(\frac{|w|^q}{q} + w\cdot \int_{\R^d}xd\mu(x) +c\right).
$$
\end{ex}

\begin{ex}
Although Proposition \ref{SpecForm} requires that $V$ grow no more than $|x|^r$ ($1\le r<p$) for $|x|$ large, the assertion is still valid in the case
$V(x)=-\frac{|x|^p}{p}$. This follows because the associated action
$$
u(x)=|a|^{q-2}a\frac{|x|^p}{p}
$$
is $C^1$ with smooth flow map $\Psi(x,t)=xe^{-(|a|^{q-2}a)t}$; here $a$ is the unique positive solution of the equation
$$
\delta a +(p-1)|a|^q -1 =0.
$$
It can be checked in this case that the associated generalized value function is given by
$$
\cU(\mu) = \frac{|a|^{q-2}a}{p}\int_{\R^d}|x|^pd\mu(x).
$$
\end{ex}



\section{Hamilton-Jacobi equations}\label{HJESec}
This section is dedicated to the proof of the main result of this paper, Theorem \ref{mainThm}, and to the proof of a special case of Conjecture \ref{GradConj}.
Our proof Theorem \ref{mainThm} requires us to define solutions of the abstract HJE \eqref{NewHJE}
$$
\delta \cU + \frac{1}{q}||\nabla U||^q_{L^q(\mu)} +\cV(\mu) = 0, \quad \mu\in \cMp.
$$
Naturally, this will involve the {\it tangent space}
$$
\text{Tan}_\mu\cMp:=\overline{\{|\nabla \psi|^{q-2}\nabla\psi: \psi\in C^\infty_c(\R^d)\}}^{L^p(\mu)}
$$
and the {\it cotangent space}
$$
\text{CoTan}_\mu\cMp:=\overline{\{ \nabla \psi: \psi\in C^\infty_c(\R^d)\}}^{L^q(\mu)}
$$
of $\cMp$ at a given measure $\mu$. We will also make use of the following characterization
\begin{equation}\label{TanSpaceId}
\text{Tan}_\mu\cMp = \overline{\left\{\lambda (r-\id): \lambda>0, \; (\id\times r)_{\#}\mu\in \Gamma_0(\mu,r_{\#}\mu)\right\}}^{L^p(\mu)}
\end{equation}
which is proved in Theorem 8.5.1 in \cite{AGS}.  In the spirit our previous work, we present a notion of sub- and super differential of functionals on $\cMp$. This notion is inspired
by Definition 10.1.1 of \cite{AGS} and Definition 3.1 of \cite{G}.
\begin{defn}\label{DiffDef}
$\xi\in CoTan_{\mu_0}\cMp$ belongs to the {\it superdifferential} of $\cU$ at $\mu_0$ if
$$
\cU(\mu)\le \cU(\mu_0) + \inf_{\pi\in \Gamma_0(\mu_0,\mu)}\int \xi(x)\cdot (y-x)d\pi(x,y) + o(W_p(\mu_0,\mu))
$$
as $\mu\rightarrow \mu_0$. In this case, we write $\xi\in \nabla^+\cU(\mu_0)$. Likewise, $\xi\in CoTan_{\mu_0}\cMp$ belongs to the {\it subdifferential} of $\cU$ at $\mu_0$ if
$$
\cU(\mu)\ge \cU(\mu_0) + \sup_{\pi\in \Gamma_0(\mu_0,\mu)}\int \xi(x)\cdot (y-x)d\pi(x,y) + o(W_p(\mu_0,\mu))
$$
as $\mu\rightarrow \mu_0$. In this case, we write $\xi\in \nabla^-\cU(\mu_0)$.
\end{defn}
It is easy to verify that if both $\nabla^+\cU(\mu)$ and $\nabla^-\cU(\mu)$ are nonempty, then they must contain a common, single element which we denote
$\nabla \cU(\mu)$ and call the derivative of $\cU$ at $\mu$.  See Remark 3.2 of \cite{G} for more on this comment.  We are now ready to define an appropriate
type of solution to the HJE \eqref{NewHJE}. The following definition originates in the work of Gangbo, Nguyen, and Tudorascu \cite{G}.
\begin{defn}\label{ViscDef}
$\cU\in USC(\cMp)$ is a {\it viscosity subsolution} of \eqref{NewHJE} if for all $\xi\in \nabla^+\cU(\mu_0)$,
\begin{equation}\label{Subsolnproperty}
\delta \cU(\mu_0) + \frac{1}{q}||\xi||^q_{L^q(\mu_0)} +\cV(\mu_0)\le 0.
\end{equation}
$\cU\in LSC(\cMp)$ is a {\it viscosity supersolution} of \eqref{NewHJE} if for all $\xi\in \nabla^-\cU(\mu_0)$,
\begin{equation}\label{Supsolnproperty}
\delta \cU(\mu_0) + \frac{1}{q}||\xi||^q_{L^q(\mu_0)} +\cV(\mu_0)\ge 0.
\end{equation}
Finally, $\cU\in C(\cMp)$ is a {\it viscosity solution} if it is both a sub- and supersolution.
\end{defn}


\begin{proof} (of Theorem \ref{mainThm})
1. Suppose $\xi\in \nabla^+\cU(\mu_0)$. For $\lambda>0$ and $r:\R^d\rightarrow \R^d$ such that $(\id\times r)_{\#}\mu_0\in \Gamma_0(\mu_0,r_{\#} \mu_0)$, set
$$
v:=\lambda (r-\id).
$$
We know such $v$ is $L^p(\mu_0)$ dense in $Tan_{\mu_0}\cMp$ by \eqref{TanSpaceId}.

\par Define the path $\sigma(t):=(\id + t v)_{\#}\mu_0$, and notice
$$
\sigma(t)=((1-t\lambda )\id + t \lambda r)_{\#}\mu_0=((1-t\lambda)\pi^1 + t\lambda \pi^2)_{\#}[(\id\times r)_{\#}\mu_0].
$$
Hence for $t\in [0,1/\lambda]$, $\sigma$ is a constant speed geodesic joining $\mu_0$ to $r_\#\mu_0$ and
$$
||\dot\sigma(t)||=||v||_{L^p(\mu_0)}, \quad t\in (0,1/\lambda).
$$
By dynamic programming (Proposition \ref{DPPprof}), for each $0<h<1/\lambda$
\begin{align*}
\cU(\mu_0)&\le e^{-\delta h}\cU(\sigma(h)) + \int^h_0e^{-\delta t}\left(\frac{1}{p}||\dot\sigma(t)||^p -\cV(\sigma(t))\right)dt \\
& = e^{-\delta h} \cU((\id + h v)_\#\mu_0) + \left(\frac{e^{-\delta h}-1}{\delta}\right) \frac{||v||^p_{L^p(\mu_0)}}{p} - \int^h_0e^{-\delta t} \cV(\sigma(t))dt.
\end{align*}
Moreover, as
$$
\pi(t)=(\id\times(\id + t v))_{\#}\mu_0\in\Gamma_0(\mu_0,\sigma(t))
$$
for all sufficiently small $t>0$,
\begin{align*}
\cU(\mu_0) &\le e^{-\delta h}\left\{ \cU(\mu_0) + h\int_{\R^d}v\cdot \xi d\mu_0 + o(h)\right\} + \\
& +\left(\frac{e^{-\delta h}-1}{\delta}\right) \frac{||v||^p_{L^p(\mu_0)}}{p} - \int^h_0e^{-\delta t} \cV(\sigma(t))dt.
\end{align*}
Hence,
$$
\delta U(\mu_0) - \int_{\R^d}v\cdot \xi d\mu_0 -\frac{||v||^p_{L^p(\mu_0)}}{p} + \cV(\mu_0)\le o(1)
$$
as $h\rightarrow 0^+$. Sending $h$ to zero and taking the supremum over $v$ gives the desired inequality \eqref{Subsolnproperty}.

\par 2. Now suppose that  $\xi\in \nabla^-\cU(\mu_0)$. Fix $\eta>0$ and observe that for each $h>0$ there is
$\sigma_h$ admissible for $\cU(\mu_0)$ such that
$$
\cU(\mu_0) > -\eta h + \int^\infty_0e^{-\delta t}\left(\frac{1}{p}||\dot\sigma_h(t)||^p -\cV(\sigma_h(t))\right)dt.
$$
A routine computation shows
$$
\cU(\mu_0) > -\eta h + e^{-\delta h}\cU(\sigma_h(h)) + \int^h_0e^{-\delta t}\left(\frac{1}{p}||\dot\sigma_h(t)||^p -\cV(\sigma_h(t))\right)dt
$$
and employing \eqref{Vgrowth} and \eqref{cond18}, we deduce
$$
 \int^h_0e^{-\delta t}||\dot\sigma_h(t)||^pdt\le C
$$
for all $h>0$. Note $C$ is a universal constant independent of $h>0$. The following uniform estimate
\begin{equation}\label{Badhest}
W_p(\sigma_h(h),\mu)\le Ch^{1-1/p},
\end{equation}
is now immediate.

\par Let us now improve upon the estimate \eqref{Badhest}.  By our computations above and the assumption that $\xi$ belongs to the subdifferential of $\cU$ at $\mu_0$,
for any $\pi_h\in \Gamma_0(\mu_0,\sigma_h(h))$
\begin{align}\label{part2mainproofbound}
\cU(\mu_0)&>-\eta h + e^{-\delta h}\left\{\cU(\mu_0)+\int \xi(x)\cdot (y-x)d\pi_h(x,y) + o(W_p(\mu_0,\sigma_h(h)))\right\} \nonumber \\
& \quad\quad\quad\quad + \int^h_0e^{-\delta t}\left(\frac{1}{p}||\dot\sigma_h(t)||^p -\cV(\sigma_h(t))\right)dt.
\end{align}
Further observe
$$
\int^h_0e^{-\delta t}\frac{1}{p}||\dot\sigma_h(t)||^pdt\ge e^{-\delta h}\frac{W_p(\sigma_h(h),\mu_0)^p}{ph^{p-1}}
$$
and by a version of Young's inequality
$$
\iint \xi(x)\cdot (y-x)d\pi_h(x,y)\ge -2^{q/p}\frac{||\xi||^q_{L^q(\mu_0)}}{q}h - \frac{W_p(\sigma_h(h),\mu_0)^p}{2ph^{p-1}}.
$$
Combining these bounds with \eqref{part2mainproofbound}
gives
\begin{align*}
0&> -\eta + \left(\frac{e^{-\delta h}-1}{h}\right)\cU(\mu_0) - o(1)\left[\frac{W_p(\sigma_h(h),\mu_0)}{h}\right] \\
& + \frac{1}{p}\left(e^{-\delta h} -\frac{1}{2}\right)\left[\frac{W_p(\sigma_h(h),\mu_0)}{h}\right]^p - \frac{1}{h}\int^h_0e^{-\delta t}\cV(\sigma_h(t))dt -2^{q/p}\frac{||\xi||^q_{L^q(\mu_0)}}{q}
\end{align*}
as $h\rightarrow 0^+$.  Futhermore, the uniform estimate \eqref{Badhest} implies $\frac{1}{h}\int^h_0e^{-\delta t}\cV(\sigma_h(t))dt=\cV(\mu_0)+o(1)$ as $h\rightarrow 0^+$, so we are now able to conclude
\begin{equation}\label{Betterhest}
W_{p}(\sigma_h(h), \mu_0)\le Ch.
\end{equation}
This estimate is valid for all $h>0$ small enough for some constant $C$; note the improvement over our previous estimate \eqref{Badhest}.

\par We return to \eqref{part2mainproofbound}, make use of \eqref{Betterhest}, and again apply Young's inequality to
find
\begin{align*}
\cU(\mu_0) & > - \eta h +e^{-\delta h}\left[\cU(\mu_0) - \frac{||\xi||^q_{L^q(\mu_0)}}{q}h - \frac{W_p(\sigma(h),\mu_0)^p}{ph^{p-1}} +o(h)\right] \\
& \quad\quad\quad + \int^h_0e^{-\delta t}\left(\frac{1}{p}||\dot\sigma(t)||^p -\cV(\sigma(t))\right)dt\\
&\ge -\eta h + e^{-\delta h}\cU(\mu_0) - e^{-\delta h}\frac{||\xi||^q_{L^q(\mu_0)}}{q}h +o(h) - \cV(\mu_0)h
\end{align*}
or
$$
\delta \cU(\mu_0) + \frac{||\xi||^q_{L^q(\mu_0)}}{q} +\cV(\mu_0) >-\eta +o(1).
$$
Sending, $h$ and then $\eta\rightarrow 0^+$ verifies the sought after inequality \eqref{Supsolnproperty}.
\end{proof}
We will conclude this paper with a proof of Conjecture \ref{GradConj} in the case of simple potentials \eqref{SimpleV}. To this end, we shall need a lemma.

\begin{lem}\label{DiffSimplePot}
Let $g:\R^d\rightarrow \R$ be a Lipschitz continuous function and set
$$
{\cal G}(\mu):=\int_{\R^d}g(x)d\mu(x), \quad \mu\in \cMp.
$$
If $g$ is differentiable for $\mu$ almost every $x\in \R^d$, then ${\cal G}$ is differentiable at $\mu$ (in the sense of Definition \ref{DiffDef}) and
$$
\nabla {\cal G}(\mu)=\nabla g, \quad \mu\; \text{a.e.}\; x\in \R^d.
$$
\end{lem}

\begin{proof}
First we choose a Borel measurable mapping $\zeta:\R^d\rightarrow \R^d$ that equals $\nabla g$ for $\mu$ almost every $x\in\R^d$. For instance, we may select $\zeta:=(\zeta^1,\dots,\zeta^n)$ as
follows:
$$
\zeta^i(x):=\limsup_{n\rightarrow \infty}\frac{g\left(x+\frac{e_i}{n}\right)-g(x)}{\frac{1}{n}}, \quad x\in \R^d
$$
for $i=1,\dots,n$.  And note that as $g$ is Lipschitz continuous, this choice gives $|\zeta(x)|\le \sqrt{d}\;\text{Lip}(g)$ for all $x\in \R^d$.

\par Next, define
$$
\omega(x,y):=
\begin{cases}
\left(g(y)-g(x) - \zeta(x)\cdot (y-x)\right)/|y-x|, \quad y\neq x\\
0, \quad x=y
\end{cases}.
$$
Clearly, $\omega$ is a Borel measurable on $\R^d\times\R^d$ and $|\omega(x,y)|\le(1+\sqrt{d})\text{Lip}(g)$ for all $x,y\in \R^d$. And by hypothesis,
\begin{equation}\label{LimSupAss}
\limsup_{y\rightarrow x}w(x,y):=\lim_{\epsilon \rightarrow 0^+}\sup_{y\in B_\epsilon(x)}w(x,y)=0
\end{equation}
for $\mu$ almost every $x\in \R^d$.  For each $\gamma\in \Gamma_o(\mu,\nu)$
$$
\int_{\R^d}g(y)d\nu(y)=\int_{\R^d}g(x)d\mu(x)+ \iint_{\R^d\times\R^d}\zeta(x)\cdot (y-x)d\gamma(x,y) + \iint_{\R^d\times\R^d}|x-y|w(x,y)d\gamma(x,y).
$$
Therefore, it suffices to show
$$
 \iint_{\R^d\times\R^d}w(x,y)|x-y|d\gamma(x,y)=o(W_p(\mu,\nu))
$$
as $\nu\rightarrow \mu$.

\par Assume $\{\nu_n\}_{n\in \N}$ is a sequence converging $\mu$ as $n\rightarrow \infty$, with $W_p(\nu_n,\mu)>0$ for each $n$.
We know from Remark 7.1.6 of \cite{AGS} that for any $\gamma_n\in \Gamma_o(\nu_n,\mu)$,
\begin{equation}\label{LimMeas}
\gamma_n\rightarrow (\id\times \id)_\# \mu
\end{equation}
narrowly in ${\cal P}_p(\R^d\times\R^d)$. We also have by H\"{o}lder's inequality,
\begin{equation}\label{LastLittleOh}
 \iint_{\R^d\times\R^d}w(x,y)|x-y|d\gamma_n(x,y) \le W_p(\nu_n,\mu) \left(\iint_{\R^d\times\R^d} |w(x,y)|^qd\gamma_n(x,y)\right)^{1/q}.
\end{equation}
\par Now let $\delta>0$ and choose $\epsilon>0$ so that
$$
\int_{\R^d}\sup_{y\in B_\epsilon(x)}|w(x,y)|^q d\mu(x)\le \delta.
$$
Such an $\epsilon>0$ exists by a simple application of Lebesgue's dominated convergence theorem. Also notice
\begin{align*}
\iint_{\R^d\times\R^d} |w(x,y)|^qd\gamma_n(x,y)& = \iint_{|x-y|< \epsilon}|w(x,y)|^qd\gamma_n(x,y) +\iint_{|x-y|\ge \epsilon} |w(x,y)|^qd\gamma_n(x,y) \\
& \le \int_{\R^d}\sup_{y\in B_\epsilon(x)}|w(x,y)|^qd\mu(x)\quad + \\
& \hspace{.75in} \left((1+\sqrt{d})\text{Lip}(g)\right)^q\gamma_n\left(\left\{ (x,y)\in \R^d: |x-y| \ge \epsilon\right\}\right) \\
& \le \delta + \left((1+\sqrt{d})\text{Lip}(g)\right)^q\gamma_n\left(\left\{ (x,y)\in \R^d: |x-y| \ge \epsilon\right\}\right).
\end{align*}
Thus, $\limsup_{n\rightarrow \infty}\iint_{\R^d\times\R^d} |w(x,y)|^qd\gamma_n(x,y)\le \delta$ and by \eqref{LimMeas} and \eqref{LastLittleOh},
$$
\limsup_{n\rightarrow \infty}\frac{ \iint_{\R^d\times\R^d}w(x,y)|x-y|d\gamma_n(x,y) }{ W_p(\nu_n,\mu)}\le \delta^{1/q}.
$$
This concludes the proof as $\delta$ and the sequence $\{\nu_n\}_{n\in \N}$ were chosen arbitrarily.
\end{proof}
\begin{cor}\label{SpecConj}
Assume that $\cV$ satisfies \eqref{SimpleV} with $V$ Lipschitz continuous, and also that \eqref{littleVbounds} holds. Then for any minimizing trajectory $\sigma$ for $\cU(\mu)$,
with minimal velocity $v$,
$$
-|v(t)|^{p-2}v(t)=\nabla\cU(\sigma(t))
$$
for Lebesgue almost every $t>0$.
\end{cor}

\begin{proof}
By our assumptions, the corresponding classical value function $u$
\eqref{classVal} with potential $V$ also Lipschitz continuous. The
proof of Corollary \ref{CompareEqLemma} yields that $u$ is
differentiable at $\sigma(t)$ almost every $x\in \R^d$ with
$\nabla u= -|v(t)|^{p-2}v(t)$ for Lebesgue almost every $t>0$. The
desired conclusion now follows from Lemma \ref{DiffSimplePot} and
Proposition \ref{SpecForm}.
\end{proof}

\appendix

\section{Probabilistic representation}
This appendix is dedicated to proving a probabilistic representation of $AC_{p,\delta}(\cMp)$ paths.  The foundation of this
assertion is the following compactness result for paths in $AC_{p,\delta}(\R^d)$.  An immediate consequence of the following lemma is that
when $V$ satisfies the bound \eqref{littleVbounds}, the classical value function $u=u(x)$ \eqref{classVal} has minimizing paths for every $x\in \R^d$.
\begin{lem}\label{ClassCompactness}
Assume the sequence $\{\gamma_k\}_{k\in \N}\subset AC_{p,\delta}(\R^d)$ satisfies
$$
\sup_{k\in \N}|\gamma_k(0)|<\infty
$$
and
$$
\sup_{k\in \N}\int^\infty_0 e^{-\delta t}|\dot\gamma_k(t)|^pdt<\infty.
$$
Then there is a subsequence $\{\gamma_{k_j}\}_{j\in \N}$ and $\gamma \in AC_{p,\delta}(\R^d)$ such that $\gamma_{k_j}\rightarrow \gamma$ locally uniformly on $[0,\infty)$ and
$$
\lim_{j\rightarrow \infty}\int^\infty_0|\gamma_{k_j}(t)-\gamma(t)|^re^{-\delta t}dt=0
$$
for each $1\le r<p$.
\end{lem}

\begin{proof}
The proof follows closely with the argument given in the proof of Lemma \ref{BlowupLem}, so we omit the details.
\end{proof}
Recall that for $\gamma,\xi\in C([0,\infty);\R^d)$, we defined
$$
\Pi(\gamma,\xi)=\sum^\infty_{k=0}\frac{1}{2^k}\frac{\max_{0\le t\le k}|\gamma(t)-\xi(t)|}{1+\max_{0\le t\le k}|\gamma(t)-\xi(t)|}
$$
in \eqref{PiMetricDef}.  It is immediate that $\Pi$ is a metric for $C([0,\infty);\R^d)$ and that $\lim_{m\rightarrow 0}\Pi(\gamma_m,\gamma)=0$ if and only if
$\gamma_m\rightarrow \gamma$ locally uniformly on $[0,\infty)$.  We argue below that $C([0,\infty);\R^d)$ is in fact a Polish space under
this metric; this claim is of interest because $C([0,\infty);\R^d)$ is not separable under the maximum norm.
\begin{prop}\label{PiMetric} The metric space
$$
\Gamma:=\left(C([0,\infty);\R^d),\Pi\right)
$$
is complete and separable.
\end{prop}
\begin{proof}
Let $\{\gamma_m\}\subset \Gamma$ be a Cauchy sequence, assume $k\in \N$ and fix $\epsilon\in (0,1)$.  Choose an integer $N$ so large that
$$
\Pi(\gamma_m, \gamma_n)<\frac{\epsilon}{2^k},\quad m,n\ge N.
$$
Notice that for $t\in [0,k]$
$$
\frac{1}{2^k}\frac{|\gamma_m(t)-\gamma_n(t)|}{1+|\gamma_m(t)-\gamma_n(t)|}\le \frac{1}{2^k}\frac{\max_{0\le s\le k}|\gamma_m(s)-\gamma_n(s)|}{1+\max_{0\le s\le k}|\gamma_m(s)-\gamma_n(s)|}<\frac{\epsilon}{2^k}
$$
from which we deduce
$$
|\gamma_m(t)-\gamma_n(t)|<\epsilon, \quad t\in [0,k],\;\; m,n\ge N.
$$
Hence $\{\gamma_m\}_{m\in\N}\subset C([0,k],\R^d)$ is Cauchy for each $k$ and so converges locally uniformly to some function $\gamma\in\Gamma$. Thus, $\Gamma$ is complete.

\par By the Weierstrass Approximation Theorem, $C([0,k],\R^d)$ equipped with maximum norm is separable for each $k\in \N$. Let $S_k:=\{h^k_1,h^k_2,h^k_3,\dots\}\subset C([0,k],\R^d)$
be dense and extend each function in this family continuously on $[0,\infty)$
$$
f^k_j(t):=
\begin{cases}
h^k_j(t), \quad 0\le t\le k \\
h^k_j(k), \quad k\le t<\infty
\end{cases}
$$
$j\in \N.$ We'll now show that $\cup_{k\in \N}S_k$ is dense in $\Gamma$.

\par Assume $\gamma\in \Gamma$ and $\epsilon>0$. Choose $m\in\N$ so large that
$$
\frac{1}{2^m}<\frac{\epsilon}{2},
$$
and $f^m_j\in \cup_{k\in \N}S_k$ such that
$$
\max_{0\le t\le m}|\gamma(t)-f^m_j(t)|<\epsilon/2.
$$
By the definition of $\Pi$,
\begin{align*}
\Pi(\gamma,f^m_j)&=\sum^m_{k=0}\frac{1}{2^k}\frac{\max_{0\le t\le k}|\gamma(t)-f^m_j(t)|}{1+\max_{0\le t\le k}|\gamma(t)-f^m_j(t)|}\\
& \quad + \sum^\infty_{k=m+1}\frac{1}{2^k}\frac{\max_{0\le t\le k}|\gamma(t)-f^m_j(t)|}{1+\max_{0\le t\le k}|\gamma(t)-f^m_j(t)|}\\
&\le \max_{0\le t\le m}|\gamma(t)-f^m_j(t)| + \frac{1}{2^m}\\
&<\epsilon.
\end{align*}
Hence, $\Gamma$ is separable.
\end{proof}

\begin{thm}\label{ProbRep}
Let $\sigma\in AC_{p,\delta}(\cMp)$ for some $1<p<\infty$ and $\delta>0$, and $v$ its minimal velocity. Then there is a Borel probability measure $\eta$ on
$\R^d\times \Gamma$ that is concentrated on $(x,\gamma)$ such that
$$
\begin{cases}
\dot\gamma(t)=v(\gamma(t),t), \quad \text{a.e. $t>0$}\\
\gamma(0)=x
\end{cases}
$$
and
$$
\sigma(t)=e(t)_\#\eta, \quad t\ge 0.
$$
\end{thm}
\begin{proof} We adapt the proof of Theorem 8.2.1 of \cite{AGS}. Recall that $v$ satisfies \eqref{vdeltapbound}, and so
$$
\int^T_0\int_{\R^d}|v(x,t)|^pd\sigma_t(x)dt<\infty
$$
for each $T>0$.  From this inequality, Proposition 8.17 and 8.18 of \cite{AGS} imply that $\sigma^\epsilon(t):=\rho^\epsilon*\sigma(t)$ and
$v^\epsilon(t):=\rho^\epsilon*(v(t)\sigma(t))/\sigma^\epsilon(t)$ satisfy the continuity equation,
$$
\int_{\R^d}|v(x,t)|^p d\sigma_t(x)\le \int_{\R^d}|v^\epsilon(x,t)|^p d\sigma^\epsilon_t(x)\quad t>0
$$
and that
$$
\sigma^\epsilon(t)=\varphi^\epsilon(t)_{\#}\sigma^\epsilon(0), \quad t\ge 0.
$$
Here $\{\rho^\epsilon\}_{\epsilon>0}$ is a family of positive mollifiers and the map $\varphi^\epsilon:[0,\infty)\times\R^d \rightarrow \R^d$ is the flow of the ODE
$$
\begin{cases}
\partial_t \varphi^\epsilon(t,x)=v^\epsilon(\varphi^\epsilon(t,x),t), \quad t>0\\
\;\;\varphi^\epsilon(0,x)=x
\end{cases}.
$$
Further denoting $\varphi^\epsilon: x\mapsto \{\varphi^\epsilon(t,x)\in \R^d: \; t\ge 0\}$, we define the measure on  $\R^d\times\Gamma$
$$
\eta^\epsilon:=(\id\times \varphi^\epsilon)_{\#}\sigma^\epsilon(0)
$$
for $\epsilon>0$.  We claim that $\{\eta^\epsilon\}_{\epsilon>0}$ is tight.

\par Let $r^1:\R^d\times \Gamma \rightarrow \R^d; (x,\gamma)\mapsto x$ and $r^2:\R^d\times \Gamma \rightarrow \R^d; (x,\gamma)\mapsto \gamma-x$. Further define
$r:=r^1\times r^2$ as mapping of $\R^d\times\Gamma \rightarrow \R^d\times\Gamma $. Note that $r\in C(\R^d\times\Gamma )$ with inverse
$$
r^{-1}(x,\gamma)=(x,\gamma+x).
$$
In particular, $r$ is a proper mapping. Further observe that $r^1_{\#}\eta^\epsilon=\sigma^\epsilon(0)=\rho^\epsilon*\sigma(0)$ converges to $\sigma(0)$ in $\cMp$ and therefore is a
tight sequence. Also defining
$$
\Xi(\gamma):=
\begin{cases}
\int^\infty_0e^{-\delta t}|\dot\gamma(t)|^pdt, \quad \gamma \in AC_{p,\delta}(\R^d)\; \text{and}\; \gamma(0)=0\\
+\infty, \quad \text{otherwise}
\end{cases}
$$
we have by our above arguments and Tonelli's theorem
\begin{align*}
\int_{\Gamma}\Xi(\gamma)d(r^2_{\#}\eta^\epsilon)(\gamma)&=\int_{\R^d}\Xi(\varphi^\epsilon(x))d((\varphi^\epsilon-\id)_{\#}\sigma^\epsilon(0))(x) \\
&= \int_{\R^d}\int^\infty_{0}e^{-\delta t}|\partial_t\varphi^\epsilon(t,x)|^pdt d\sigma^\epsilon_0(x)\\
&=\int_{\R^d}\int^\infty_{0}e^{-\delta t}|v^\epsilon(\varphi^\epsilon(t,x),t)|^pdt d\sigma^\epsilon_0(x)\\
&=\int^\infty_{0}\left\{\int_{\R^d}e^{-\delta t}|v^\epsilon(\varphi^\epsilon(t,x),t)|^p d\sigma^\epsilon_0(x)\right\}dt\\
&=\int^\infty_{0}\left\{\int_{\R^d}e^{-\delta t}|v^\epsilon(x,t)|^p d\sigma^\epsilon_t(x)\right\}dt\\
&\le\int^\infty_{0}\left\{\int_{\R^d}e^{-\delta t}|v(x,t)|^p d\sigma_t(x)\right\}dt\\
&<\infty .
\end{align*}
Since the sublevel sets of $\Xi$ are compact in $\Gamma$ by Lemma
\ref{ClassCompactness}, $\{r^2_{\#}\eta^\epsilon\}_{\epsilon>0}$
is tight. By Lemma 5.2.2 of \cite{AGS},
$\{\eta_{\epsilon}\}_{\epsilon>0}$ is consequently tight, as well.

\par It follows that there is a sequence of positive numbers $\epsilon_k$ tending to $0$ as $k\rightarrow \infty$
and a probability measure $\eta$ on $\R^d\times \Gamma$ such that $\eta^{\epsilon_k}\rightarrow \eta$ narrowly in $\R^d\times \Gamma$ as $k\rightarrow\infty$.  Moreover, for each $f\in C_b(\R^d)$ and $t\ge 0$
\begin{align*}
\int_{\R^d}f(x)d\sigma_t(x)&=\lim_{k\rightarrow \infty}\int_{\R^d}f(x)d\sigma^{\epsilon_k}_t(x) \\
&=\lim_{k\rightarrow \infty}\int_{\R^d\times \Gamma}f(\gamma(t))d\eta^{\epsilon_k}(x,\gamma) \\
&=\int_{\R^d\times \Gamma}f(\gamma(t))d\eta(x,\gamma).
\end{align*}
As a result $\sigma(t)=e(t)_{\#}\eta$, as claimed. By a slight modification of the argument given in Theorem 8.2.1 of \cite{AGS}, we conclude that $\eta$ is concentrated on $(x,\gamma)$ for which
$\dot\gamma(t)=v(\gamma(t),t)$ for almost every $t>0$ and $\gamma(0)=x$.
\end{proof}
\begin{rem}\label{FubiniRemark}
For the $\sigma$ and $\eta$ as in the above claim, Tonelli's theorem implies
\begin{align*}
\int_{\R^d\times \Gamma}\left\{\int^\infty_0 e^{-\delta t}|\dot\gamma(t)|^pdt\right\}d\eta(x,\gamma) &= \int^\infty_0 \int_{\R^d\times \Gamma}e^{-\delta t}|\dot\gamma(t)|^pd\eta(x,\gamma)dt \\
&=\int^\infty_0 \int_{\R^d\times \Gamma}e^{-\delta t}|v(\gamma(t),t)|^pd\eta(x,\gamma)dt \\
&=\int^\infty_0 \int_{\R^d}e^{-\delta t}|v(x,t)|^pd\sigma_t(x)dt \\
&=\int^\infty_0 e^{-\delta t}||\dot\sigma(t)||^pdt \\
&<\infty.
\end{align*}
Consequently, for $\eta$ almost every $(x,\gamma)\in \R^d\times \Gamma$, $\gamma\in AC_{p,\delta}(\R^d)$.
\end{rem}


\end{document}